\documentclass[11pt,letterpaper]{article}
\usepackage[margin=1.25in]{geometry}
\usepackage{amsmath, amssymb,amsthm, xcolor}
\usepackage{mathabx}
\usepackage{bbm}
\usepackage{tikz}
\usetikzlibrary{patterns}

\title{Characterizing stationary 1+1 dimensional lattice polymer models}
\author{Hans Chaumont\thanks{Department of Mathematics, University of Wisconsin. Madison, WI, USA.\newline \href{mailto:chaumont@wisc.edu}{chaumont@wisc.edu}} \and Christian Noack\thanks{Department of Mathematics, University of Wisconsin. Madison, WI, USA.\newline \href{mailto:cnoack@wisc.edu}{cnoack@wisc.edu}}}
\usepackage{hyperref}
\newtheorem{theorem}{Theorem}[section]
\newtheorem{corollary}[theorem]{Corollary}
\newtheorem{lemma}[theorem]{Lemma}
\newtheorem{definition}[theorem]{Definition}
\newtheorem{proposition}[theorem]{Proposition}
\theoremstyle{remark}
\newtheorem{remark}[theorem]{Remark}

\numberwithin{equation}{section}

\newcommand{\p}{\partial} %Partial 
\newcommand{\Z}{\mathbb{Z}}

\newcommand{\N}{\mathbb{N}}
\newcommand{\R}{\mathbb{R}}

\renewcommand{\t}{\widetilde}

\newcommand{\Os}{O_1\times O_2\times O_3}

\newcommand{\scaleOs}{c_1O_1\times c_2O_2\times c_1O_3}
\newcommand{\ID}{(R^1,R^2,Y)}
\newcommand{\flipID}{\left(R^2,R^1,h(Y)\right)}
\newcommand{\scaleID}{(c_1R^1,c_2R^2,c_1Y)}

\begin{document}
\maketitle

\begin{abstract}
Motivated by the study of directed polymer models with random weights on the square integer lattice, we define an integrability property shared by the log-gamma, strict-weak, beta, and inverse-beta models. This integrability property encapsulates a preservation in distribution of ratios of partition functions which in turn implies the so called Burke property. We show that under some regularity assumptions, up to trivial modifications, there exist no other models possessing this property.
\end{abstract}



\section{Introduction}

One method which has been used to study certain models of percolation and polymers is to introduce a version of the model with boundary conditions that possesses a stationarity property. This stationarity property allows for the exact computation of some quantities of interest, such as the free energy. In \cite{OY2001} O'Connell and Yor introduce a model for a directed polymer in a Brownian environment with a Burke-type stationarity property.  In \cite{SV2010} Sepp{\"a}l{\"a}inen and Valk\'o use this stationarity to find bounds on the fluctuation exponents of the free energy and the fluctuation of the paths. In \cite{CG2006} Cator and Groeneboom  relate a stationary version of the Hammersley process to the location of a second class particle and determine the order of the variance of the longest weakly north-east path. In \cite{BCS2006} Bal\'azs, Cator, and Sepp\"al\"ainen use a stationary version of the last passage growth model with exponential weights to study the variance of the last passage time and transversal fluctuations of the maximal path. 

We define the integrability property $T^{h,Y}$-invariance (Definition \ref{definition Thy invariance}) which encapsulates this stationarity in the setting of lattice directed polymers. This property implies a preservation in distribution of ratios of partition functions. The first model discovered possessing this property is the log-gamma model, introduced by Sepp{\"a}l{\"a}inen in \cite{S2012}.  In his paper $T^{h,Y}$-invariance is used to prove the conjectured values for the fluctuation exponents of the free energy and the polymer path in the stationary point-to-point case and to prove upper bounds for the exponents in the point-to-point and point-to-line cases without boundary conditions. In \cite{GS2013} Georgiou and Sepp\"al\"ainen use $T^{h,Y}$-invariance to obtain large deviation results for the log-gamma polymer. In the setting of directed polymer models, this is the first instance where precise large deviation rate functions for the free energy were derived.

Thereafter three additional models admitting $T^{h,Y}$-invariant versions were found: the strict-weak model, introduced simultaneously by Corwin, Sepp{\"a}l{\"a}inen, and Shen in \cite{CSS2015} and O'Connell and Ortmann in \cite{OO2015}, the beta model, introduced by Barraquand and Corwin in \cite{BC2016} as the beta RWRE, and the inverse-beta model, introduced by Thiery and Le Doussal in \cite{TL2015}. The stationary versions of these models were given by Bal\'azs, Rassoul-Agha, and Sepp{\"a}l{\"a}inen in \cite{BRS2016} for the beta model, Thiery in \cite{T2016} for the inverse-beta model, and by Corwin, Sepp{\"a}l{\"a}inen, and Shen in \cite{CSS2015} for the strict-weak model.

In this paper we present a uniqueness result for $T^{h,Y}$-invariant models. That is, under some regularity assumptions and up to the two natural modifications of reflection and scaling, the log-gamma, strict-weak, beta, and inverse-beta are the only $T^{h,Y}$-invariant models.

In the forthcoming paper \cite{CN} we use $T^{h,Y}$-invariance along with a Mellin transform framework to simultaneously prove the conjectured value for the fluctuation exponent of the free energy and the upper bound for the polymer path fluctuations in the stationary point-to-point version of these four models.

\subsection{The polymer model}
The directed polymer in a random environment, first introduced by Huse and Henley \cite{HH1985}, models a long chain of molecules in the presence of random impurities. Imbrie and Spencer \cite{IS1988} formulated this model as a random walk in a random environment. See the lectures by Comets \cite{Comets2017} for a survey of results on directed polymers. We consider a class of 1+1-dimensional directed polymers on the integer lattice.

Notation: $\N= \{1,2,\ldots\}$, $\Z_+ = \{0,1,\ldots\}$, and $\R$ denotes the real numbers.
%The setup in this paper is similar to that presented in \cite{TL2015}.
On each edge $e$ of the $\Z_+^2$  lattice we place a positive random weight. For $x\in \N^2$, let $u_x$ and $v_x$ denote the horizontal and vertical incoming edge weights. We assume that the collection of pairs $\{(u_x,v_x)\}_{x\in \N^2}$ is independent and identically distributed, but do not insist that $u_x$ is independent of $v_x$ (in fact we will later assume $v_x$ is a function of $u_x$). Call this collection the \emph{bulk weights}. For $x\in \N\times \{0\}$, let $R^1_x$ denote the horizontal incoming edge weight, and for $x\in \{0\}\times \N$, let $R^2_x$ denote the vertical incoming edge weight. We assume the collections $\{R^1_x\}_{x\in \N\times \{0\}}$ and $ \{R^2_y\}_{y\in \{0\}\times \N}$ are independent and identically distributed, and refer to them as the \emph{horizontal} and \emph{vertical} \emph{boundary weights}, respectively. We further assume that the horizontal boundary weights, the vertical boundary weights, and the bulk weights are independent of each other. This assignment of edge weights is illustrated in Figure \ref{fig-weights}.

\begin{figure}[ht]
  \centering
  \begin{tikzpicture}
    
   \draw [help lines] (0,0) grid (5,5);	  
   \draw [thick, <->] (0,5) -- (0,0) -- (5,0);
   \draw [thick] (2,3) -- (3,3) -- (3,2);
   \node [left,scale=.9] at (0,2.5) {$R^2_{0,j}$};
   \node [below,scale=.9] at (2.5,0) {$R^1_{i,0}$};
   \node [above,scale=.9] at (2.5,3) {$u_x$};
   \node [right,scale=.9] at (3,2.5) {$v_x$};
   \draw[fill] (3,3) circle [radius=.05];
   \node [above right] at (2.91,2.91) {$x$};

  \end{tikzpicture}
    \caption{Assignment of edge weights.}
  \label{fig-weights}
\end{figure}
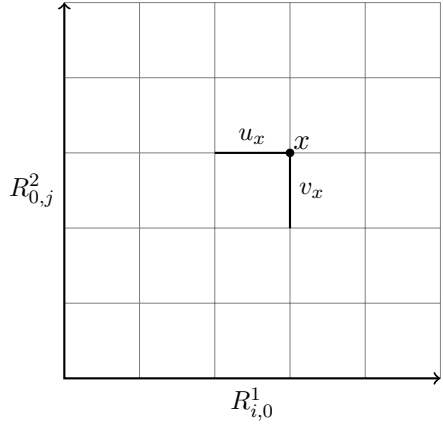

For $(m,n)\in \Z^2_+\setminus \{(0,0)\}$, let $\Pi_{m,n}$ be the collection of all up-right paths from $(0,0)$ to $(m,n)$. See Figure \ref{up-right path} for an example of such a path. We identify paths $x_\centerdot = (x_0, x_1, \ldots, x_{m+n})$ by
their sequence of vertices, but also associate to paths their sequence
of edges $(e_1, \ldots, e_{m+n})$, where $e_i = \{x_{i-1}, x_i\}$. The point-to-point partition function for the directed polymer is defined as
\[ Z_{m,n} :=
\sum_{x_\centerdot \in \Pi_{m,n}} \prod_{i=1}^{m+n}
\omega_{e_i} \qquad \text{for } (m,n) \in Z_+^2\setminus\{(0,0)\},\] 
where $\omega_e$ is the weight associated to the edge $e$. At the origin, define $Z_{0,0}:=1$.
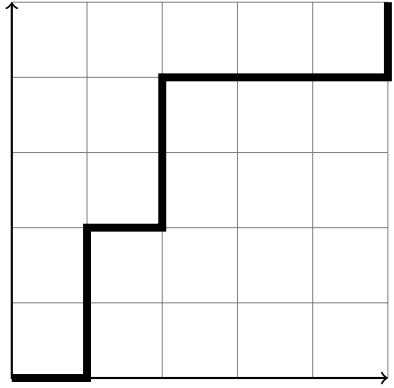
\begin{figure}[ht]
\centering
	\begin{tikzpicture}
  		\draw [help lines] (0,0) grid (5,5);	  
   		\draw [thick, <->] (0,5) -- (0,0) -- (5,0);
   		\draw [line width=3pt] (0,0) -- (1,0) -- (1,1) -- (1,2) -- (2,2) -- (2,3) -- (2,4) -- (3,4)--(4,4) -- (5,4) -- (5,5);
	\end{tikzpicture}
	\caption{An up-right path from $(0,0)$ to $(5,5)$.}
    \label{up-right path}
\end{figure}

\noindent Write $\alpha_1=(1,0)$, $\alpha_2=(0,1)$. The partition functions satisfy the recurrence relation
\begin{equation}
Z_x=u_xZ_{x-\alpha_1}+v_x Z_{x-\alpha_2} \qquad \text{for } x\in\mathbb{N}^2.\label{rec0}
\end{equation}
For $k=1,2$ define ratios of partition functions
\begin{align*}
R^k_x:=\frac{Z_x}{Z_{x-\alpha_k}}\qquad \text{for all  }x\text{ such that }  x-\alpha_k\in\mathbb{Z}_+^2.
\end{align*}
Note that these extend the definitions of $R^1_{i,0}$ and $R^2_{0,j}$, since for example $Z_{i,0}=\prod_{k=1}^{i} R^1_{k,0}$.
The recurrence relation \eqref{rec0} yields the recursions
\begin{align}
\begin{split}
R^1_x &= u_x+v_x\frac{R^1_{x-\alpha_2}}{R^2_{x-\alpha_1}}\\
R^2_x &= u_x\frac{R^2_{x-\alpha_1}}{R^1_{x-\alpha_2}}+v_x
\end{split}
\qquad \text{for } x\in\mathbb{N}^2.
\label{recursions}
\end{align}

We look to exploit these recursions to obtain more structure of the ratios $R^1_x$ and $R^2_x$, which in turn allows us to analyze quantities of interest such as the free energy, $\log Z_{m,n}$. The notation $X \stackrel{d}{=} Y$ is used to specify that random vectors $X$ and $Y$ have the same distribution. We look for cases where $(R^1_x,R^2_x)\stackrel{d}{=}(R^1_{x-\alpha_2},R^2_{x-\alpha_1})$, under the assumption that $u_x$ and $v_x$ have a functional dependence of the form $(u_x,v_x)=\big(Y_x,h(Y_x)\big)$ for some positive random variable $Y_x$ and positive function $h$. We further assume that there exist positive random variables $R^1,R^2,Y$ such that the horizontal boundary weights, the vertical boundary weights, and the bulk weights are distributed as $R^1$, $R^2$, and $\big(Y, h(Y)\big)$, respectively.

When $Y$ is a random variable taking values in the domain of $h$ and $(R^1, R^2)$ is a random vector taking values in $(0,\infty)^2$, define the random vector
\begin{equation}\label{eq-ThY}
T^{h,Y}(R^1,R^2):=\big(Y+h(Y)\frac{R^1}{R^2},\, Y\frac{R^2}{R^1}+h(Y)\big).  
\end{equation}
Note that with $(u_x, v_x) = \big(Y_x, h(Y_x)\big)$, the recursive equations \eqref{recursions} imply
\begin{equation}
(R^1_x,R^2_x)=T^{h,Y_x}(R^1_{x-\alpha_2},R^2_{x-\alpha_1})\qquad \text{for all } x\in\mathbb{N}^2.\label{rec3}
\end{equation}
\begin{definition}\label{definition Thy invariance}
Let $O_3\subset (0,\infty)$, $h: O_3 \to (0,\infty)$, and assume the random variable $Y$ takes values in $O_3$. Let $(R^1, R^2)$ be a random vector taking values in $(0,\infty)^2$ that is independent of $Y$. We say that $(R^1, R^2)$ is {\rm $T^{h,Y}$-invariant} if $T^{h,Y}(R^1,R^2) \stackrel{d}{=} (R^1, R^2)$.
\end{definition}
\noindent Definition \ref{definition Thy invariance}, while stated in terms of the random variables $(R^1, R^2)$ and $Y$, is really a property of the distributions of $(R^1, R^2)$ and $Y$.

If $(R^1,R^2)$ is $T^{h,Y}$-invariant with $R^1$ independent of $R^2$, then \eqref{rec3} and an induction argument imply that the polymer model possesses a form of stationarity:
\begin{align}\label{equation stationarity}
(R^1_x,R^2_x)\stackrel{d}{=} (R^1,R^2)\qquad \text{for all } x\in \mathbb{N}^2.
\end{align}
Although our two main theorems require $R^1$ and $R^2$ to be independent, the results in Section \ref{section equivalences} hold without this independence.
\subsection{Main results}

Our first main result, Theorem \ref{thm-h-linear}, consists of showing that, under some regularity assumptions, $T^{h,Y}$-invariance can only occur if $h$ is of the form $h(y)=a+by$ for real numbers $a,b$ satisfying $a\vee b>0$.  Our second main result, Theorem \ref{thm-classify-linear}, consists of showing that if $h$ has this form, then $T^{h,Y}$-invariance  only arises as a modification of the four known invariant models (described in \eqref{model-IG} through \eqref{model-IB}).

Given a real valued function $f$ we call  $\{x: f(x) \neq 0\}$ the support of $f$.   Note that we do not insist on taking the closure of this set. Define the non-random analogue of \eqref{eq-ThY}, 
\begin{equation}\label{equation - Thy}
T^{h,y}(r_1,r_2) := \big(y + h(y) \tfrac{r_1} {r_2}, y\tfrac{r_2}{r_1} + h(y)\big).\end{equation}

\begin{theorem}\label{thm-h-linear}
Let $R^1,R^2,Y$ be positive, independent random variables with respective densities $f_1, f_2, f_3$. Assume that the support of $f_j$ is $O_j\subset (0,\infty)$ for $j=1,2,3$, where each $O_j$ is open and $O_3$ is connected. Assume $f_1,f_2$ are twice differentiable on $O_1$ and $O_2$ respectively and that $f_3$ is three times differentiable on $O_3$. Suppose $h:O_3\to (0,\infty)$ is  four times differentiable, the mapping $\Os \owns (r_1,r_2,y)\mapsto T^{h,y}(r_1,r_2)$ surjects onto $O_1\times O_2$, and $\frac{r_2}{r_1} + h'(y) \neq 0$ for all $(r_1,r_2, y) \in \Os$. If $(R_1,R_2)$ is $T^{h,Y}$-invariant, then $h$ must be of the form $h(y) = a+by$, where $a,b$ are real numbers satisfying $a\vee b>0$.
\end{theorem}

\begin{remark}
If $(R^1,R^2,Y)$ has support $\Os$ and $(R^1,R^2)$ is $T^{h,Y}$-invariant, then the surjectivity condition is a natural assumption. As an example for when the assumption on $\frac{r_2}{r_1} + h'(y) \neq 0$ is satisfied, we can take $h$ to be any differentiable increasing function. 
Note that the assumptions do not require $O_1$ or $O_2$ to be connected.
\end{remark}
Before giving the second main result we give the form of each of the four known invariant models.

The notation $X\sim \text{Ga}(\alpha,\beta)$ is used to denote that a random variable is gamma$(\alpha,\beta)$ distributed, i.e.\ has density $\Gamma(\alpha)^{-1} \beta^\alpha x^{\alpha-1}e^{-\beta x}$ supported on $(0,\infty)$, where $\Gamma(\alpha) = \int_0^\infty x^{\alpha-1} e^{-x} dx$ is the gamma function. $X\sim \text{Be}(\alpha,\beta)$ is used to say that $X$ is beta$(\alpha,\beta)$ distributed, i.e.\ has density $\frac{\Gamma(\alpha+\beta)}{\Gamma(\alpha) \Gamma(\beta)} x^{\alpha-1}(1-x)^{\beta-1} $ supported on $(0,1)$. We then use $X\sim \text{Ga}^{-1}(\alpha,\beta)$ and $X\sim \text{Be}^{-1}(\alpha,\beta)$ to denote that $X^{-1}\sim \text{Ga}(\alpha,\beta)$ and $X^{-1}\sim \text{Be}(\alpha,\beta)$, respectively. We also use $X\sim \left(\text{Be}^{-1}(\alpha,\beta) -1\right)$ to denote that $X+1\sim\text{Be}^{-1}(\alpha,\beta)$. The symbol $\otimes$ is used to denote (independent) product distribution.

\begin{itemize}
\item \textbf{Inverse-gamma}: This is also known as the log-gamma model. Assume $\mu>\lambda>0, \, \beta>0$ and
\begin{equation}
\begin{gathered}
(R^1,R^2,Y)\sim \text{Ga}^{-1}(\mu-\lambda,\beta)\otimes \text{Ga}^{-1}(\lambda,\beta)\otimes \text{Ga}^{-1}(\mu,\beta).\label{model-IG}
\end{gathered}
\end{equation}
Then $(R^1,R^2)$ is $T^{h,Y}$-invariant, where $h(y) = y$. (See Lemma 3.2 of \cite{S2012}.)
\item \textbf{Gamma}: This is also known as the strict-weak model. Assume $\lambda,\mu,\beta>0$ and 
\begin{equation}
\begin{gathered}
(R^1,R^2,Y)\sim \text{Ga}(\mu+\lambda,\beta)\otimes \text{Be}^{-1}(\lambda,\mu)\otimes \text{Ga}(\mu,\beta).
\end{gathered}\label{model-G}
\end{equation}
Then $(R^1, R^2)$ is $T^{h,Y}$-invariant, where $h(y) = 1$. (See Lemma 6.3 of \cite{CSS2015}.)
\item \textbf{Beta}: 
Assume $\lambda,\mu,\beta>0$ and 
\begin{equation}
\begin{gathered}
(R^1,R^2,Y)\sim \text{Be}(\mu+\lambda,\beta)\otimes \text{Be}^{-1}(\lambda,\mu)\otimes \text{Be}(\mu,\beta).
\end{gathered}\label{model-B}
\end{equation}
Then $(R^1,R^2)$ is $T^{h,Y}$-invariant, where $h(y) = 1-y$. (See Lemma 3.1 of \cite{BRS2016}.)
\item \textbf{Inverse-beta}:
Assume $\mu>\lambda>0,\, \beta>0$ and 
\begin{equation}
\begin{gathered}
(R^1,R^2,Y)\sim \text{Be}^{-1}(\mu-\lambda,\beta)\otimes \left(\text{Be}^{-1}(\lambda,\beta+\mu-\lambda)-1\right)\otimes \text{Be}^{-1}(\mu,\beta).
\end{gathered}\label{model-IB}
\end{equation}
Then $(R^1,R^2)$ is $T^{h,Y}$-invariant, where $h(y) = y-1$. (See Proposition 3.1 of \cite{T2016}.)
\end{itemize}
The name of each model refers to the distribution of the bulk weights. We call these models the \textbf{four basic beta-gamma models}.

\begin{theorem}\label{thm-classify-linear}
Let  $O_j\subset (0,\infty)$ for $j=1,2,3$ and assume $h:O_3\rightarrow (0,\infty)$ has the form $h(y)=a+by$, where $a,b$ are real numbers satisfying $a\vee b>0$. Assume the mapping $\Os \owns (r_1,r_2,y)\mapsto T^{h,y}(r_1,r_2)$ surjects onto $O_1\times O_2$, and $R^1$, $R^2$, $Y$ are non-degenerate, independent random variables taking values in $O_1,\,O_2,\,O_3$ respectively.
\begin{enumerate}
\item If $a=0$ and $b>0$, then $(R^1, R^2)$ is $T^{h,Y}$-invariant if and only if $\left(R^1, \frac{1}{b}R^2, Y\right)$ is distributed as in \eqref{model-IG}.
\item If  $a>0$ and $b=0$, then $(R^1, R^2)$ is $T^{h,Y}$-invariant if and only if $\left(R^1, \frac{1}{a}R^2, Y \right)$ is distributed as in \eqref{model-G}.
\item If $a>0$, $b<0$, and $-b\notin \{\frac{y}{x}:(x,y)\in O_1\times O_2\}$, then $(R^1, R^2)$ is $T^{h,Y}$-invariant if and only if either $\left(-\frac{b}{a}R^1, \frac{1}{a}R^2, -\frac{b}{a}Y\right)$ or $\left(\frac{1}{a}R^2, -\frac{b}{a}R^1, 1+\frac{b}{a}Y\right)$ is distributed as in \eqref{model-B}.
\item If $a<0$ and $b>0$, then $(R^1, R^2)$ is $T^{h,Y}$-invariant if and only if $\left( -\frac{b}{a}R^1,-\frac{1}{a} R^2, -\frac{b}{a} Y \right)$ is distributed as in \eqref{model-IB}.
\item If $a,b>0$, then $(R^1, R^2)$ is $T^{h,Y}$-invariant if and only if $\left(\frac{1}{a}R^2, \frac{b}{a}R^1, 1+ \frac{b}{a}Y\right)$ is distributed as in \eqref{model-IB}.
\end{enumerate}
\end{theorem}
Figure \ref{fig-fundamental modifications} illustrates which one of the four basic beta-gamma models corresponds to each choice of parameters $a,b$.

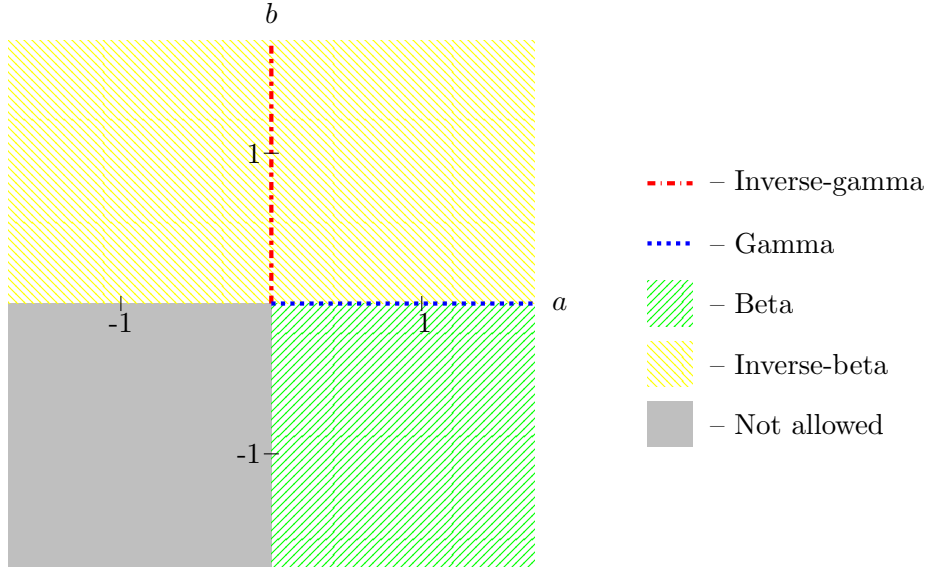
\begin{figure}[ht]
  
  \centering
  \begin{tikzpicture}[scale=1]
	  \path [fill=lightgray] (-3.5,-3.5) rectangle (0,0);
	  \path [pattern=north west lines, pattern color = yellow] (-3.5,0) rectangle (3.5,3.5);
%	  \path [fill=lime] (0,-1.75) rectangle (1.75, 0);
	  \path [pattern=north east lines, pattern color = green] (0,-3.5) rectangle (3.5, 0);

%	  \draw [thick,lightgray, <->] (-3.5,0) -- (3.5,0);
%	  \draw [thick,lightgray, <->] (0,-3.5) -- (0,3.5);
	  \draw [ultra thick, dash dot, red] (0,0) -- (0,3.5);
	  \draw [ultra thick, dotted, blue] (0,0) -- (3.5,0);

	  \node [right] at (3.6,0) {$a$};
	  \node [above] at (0,3.6) {$b$};
	  
	  \node [left] at (0,2) {1};
	  \node [below] at (-2,0) {-1};
	  \node [left] at (0,-2) {-1};
	  \node [below] at (2.04,0) {1};

% 	  \draw[fill] (0,1) circle [radius=.04];
% 	  \node [above right] at (0,1) {IG};

% 	  \draw[fill] (-1,1) circle [radius=.04];
% 	  \node[above right] at (-1,1) {IB};

% 	  \draw[fill] (1,1) circle [radius=.04];
% 	  \node[above right] at (1,1) {IB-2};
	  
% 	  \draw[fill] (1,0) circle [radius=.04];
% 	  \node[above right] at (1,0) {G};

% 	  \draw[fill] (1,-1) circle [radius=.04];
% 	  \node[above right] at (1,-1) {B, B-2};
	  %tick marks
	  \draw (-2,.1) -- (-2, -.1);
	  \draw (-.1, -2) -- (.1,-2);
      
      \draw (2,.1) -- (2, -.1);
      \draw (-.1, 2) -- (.1,2);
    
      \begin{scope}[shift = {(5,0)}]
      \draw [red, dash dot, ultra thick] (0,1.6) -- (.6,1.6);
      \node [right] at (.7,1.6) {-- Inverse-gamma};
      
      \draw [blue, dotted, ultra thick] (0,.8) -- (.6,.8);
      \node [right] at (.7,.8) {-- Gamma};
      
      \path [green, pattern = north east lines,pattern color = green] (0,-.3) rectangle (.6,.3);
      \node [right] at (.7,0) {-- Beta};
      
      \path [yellow, pattern = north west lines, pattern color = yellow] (0,-1.1) rectangle (.6,-.5);
      \node [right] at (.7,-.8) {-- Inverse-beta};
      
      \path [fill=lightgray] (0,-1.9) rectangle (.6,-1.3);
      \node [right] at (.7,-1.6) {-- Not allowed};
      \end{scope}

  \end{tikzpicture}
  \caption{Modifications of the four beta-gamma models.}
  \label{fig-fundamental modifications}
\end{figure}

In the related work \cite{TL2015}, Thiery and Le Doussal study the implications of Bethe ansatz solvability in the context of $1+1$-dimensional lattice directed polymers. In their work, they consider the model without boundary and do not impose the additional assumption that the weights on incoming horizontal and vertical edges, $u_x$ and $v_x$, have a functional dependence. Making the assumption of coordinate Bethe ansatz solvability, that is, diagonalizability of the time evolution operator of the $n$-point correlation functions, they arrive at a formula for the joint moments of $u_x$ and $v_x$. Under some additional assumptions, they conclude that the four basic beta-gamma models are the only Bethe ansatz solvable ones.

The fact that the Bethe ansatz solvable lattice polymer models (with some additional assumptions) coincide with the stationary models (with sufficient regularity) suggests a possible connection between the two integrability properties. In the current paper we do not further explore this connection, but consider it an interesting direction for future research.

\emph{Structure of the paper:} In Section \ref{section equivalences} we define the stronger property $T^h$-invariance, and give conditions for when $T^{h,Y}$-invariance is equivalent to $T^h$-invariance. $T^h$-invariance will be used as a tool in proving our main theorems. The proof of Theorem \ref{thm-h-linear} is then given in Section \ref{sec-proof of second thm}. In Section \ref{sec-modifications} we describe the natural modifications of reflection and scaling. The proof of Theorem \ref{thm-classify-linear} is given in Section \ref{sec-proof of first thm}.  

\emph{Acknowledgements:} This work is part of our dissertation research at the University of Wisconsin-Madison.  We thank our advisors Timo Sepp{\"a}l{\"a}inen and Benedek Valk\'o for their guidance and insights.

\section{Equivalences between \texorpdfstring{$T^{h,Y}$}{Thy}-invariance and \texorpdfstring{$T^h$}{Th}-invariance}
\label{section equivalences}
First define
\begin{align}\label{eq-T1T2}
 T^h_1(r_1,r_2,y) &:= y+h(y)\frac{r_1}{r_2} & T^h_2(r_1,r_2,y) &:= y\frac{r_2}{r_1}+h(y).
\end{align}
Notice that $(R^1, R^2)$ is $T^{h,Y}$-invariant if and only if 
\[
(T^h_1, T^h_2)(R^1,R^2,Y):= \big(T^h_1(R^1,R^2,Y), T^h_2(R^1,R^2,Y)\big)\stackrel{d}{=}(R^1,R^2).
\]
In this section we determine conditions which allow us to construct a function $T_3^h$ such that $(T^h_1, T^h_2,T_3^h)(R^1,R^2,Y) \stackrel{d}{=}(R^1,R^2,Y)$. Moreover, $T_3^h$ will be such that $T:=(T^h_1,T^h_2,T_3^h)$ is an involution. Recall that a function $T$ is an involution if $T\circ T$ is the identity function. 
\begin{definition}\label{def-polymer involution}
Let $O\subset (0,\infty)^2$, $O_3\subset (0,\infty)$, and $h:O_3\rightarrow (0,\infty)$. We say that an involution $T:O\times O_3 \to O\times O_3$ is a {\rm polymer involution adapted to $h$} if its first two coordinates are as in \eqref{eq-T1T2}. % When we wish to emphasize the domain, we say that $T$ is a polymer involution on $O\times O_3$.
\end{definition}

Existence and uniqueness of polymer involutions is addressed in Lemma \ref{lemma TFAE}.  When the polymer involution adapted to $h$ is unique we write $T^h$. In our two main theorems we assume that $R^1$ and $R^2$ are independent and therefore take $O = O_1 \times O_2$. We allow for arbitrary $O\subset (0,\infty)^2$ since the results in this section allow for dependence between $R^1$ and $R^2$.

\begin{definition}\label{T-Invar} Suppose $(R^1,R^2,Y)$ is a random vector taking values in $O\times O_3$, where $O\subset(0,\infty)^2$, $O_3\subset(0,\infty)$, and $Y$ is independent of $(R^1,R^2)$. Let $h:O_3\to (0,\infty)$. If there exists a polymer involution $T$ on $O\times O_3$ adapted to $h$ such that $T(R^1,R^2,Y)\stackrel{d}{=} (R^1,R^2,Y)$, then we say $(R^1,R^2,Y)$ is {\rm $T$-invariant (with respect to $h$)}.
\end{definition}

If $(R^1,R^2,Y)$ is $T$-invariant, the polymer model with weight distributions $(R^1,R^2,Y)$ not only has property \eqref{equation stationarity}, but possesses a stronger form of stationarity called the Burke property (see Theorem 3.3 of \cite{S2012}).  In Definition \ref{def-polymer involution} we restrict our attention to involutions, as $T$-invariance not only implies stationarity, but also a form of reversibility: the construction of a dual measure (see Section 3.2 of \cite{S2012} for more details).

The four basic beta-gamma models are not only $T^{h,Y}$-invariant, but are in fact $T^h$-invariant as well.  The rest of this section is dedicated to relating the properties of $T^{h,Y}$-invariance and $T^h$-invariance, as given in the following proposition.

\begin{proposition}\label{prop-invar-tinvar-eq}
Let $O\subset (0,\infty)^2$, $O_3 \subset (0,\infty)$, and $h:O_3 \to (0,\infty)$. Assume $(R^1,R^2,Y)$ is a random vector taking values in $O\times O_3$ and that $Y$ is independent of $(R^1,R^2)$. Then the following two conditions are equivalent.
\begin{enumerate}
\item The mapping $O\times O_3 \owns (r_1,r_2,y) \mapsto T^{h,y}(r_1,r_2)$ surjects onto $O$, for every $(r_1, r_2)\in O$ the function $O_3\owns y \mapsto y\frac{r_2}{r_1} + h(y)$ is injective, and $(R^1,R^2)$ is $T^{h,Y}$-invariant.
\item There exists a unique polymer involution $T^h$ adapted to $h$ on $O\times O_3$ and $(R^1,R^2,Y)$ is $T^h$-invariant.
\end{enumerate}
\end{proposition}
\noindent The proof of Proposition \ref{prop-invar-tinvar-eq} follows from combining Lemmas \ref{lemma TFAE}, \ref{lemma distr. TFAE}, and Remark \ref{remark after TFAE lemma} below.

We use the notation $\pi_j:(0,\infty)^2\rightarrow (0,\infty)$ to denote the projection onto the $j$-th coordinate for $j=1,2$.  Given $O\subset (0,\infty)^2$, $Q(O)$ will denote the set $\{\frac{y}{x}: (x,y)\in O\}$. When $O= O_1\times O_2$ we will write $\frac{O_2}{O_1}$ for $Q(O)$.

When $T$ is a polymer involution adapted to $h$ we will often use the following notation 
\begin{equation}
(\t{r}_1,\t{r}_2,\t{y}):=T(r_1,r_2,y)\label{tildes}.
\end{equation}
More precisely, by equations \eqref{eq-T1T2}
\begin{align*}
\t{r}_1 &:= y+h(y)\frac{r_1}{r_2},&
\t{r}_2 &:= y\frac{r_2}{r_1}+h(y),&
\t{y} &:= T_3^h(r_1,r_2,y).
\end{align*}
Note that these definitions imply that
\begin{equation}
\frac{\t{r}_2}{\t{r}_1}=\frac{r_2}{r_1}.\label{persistence of ratios}
\end{equation}
This equality of ratios will turn out to be quite useful.

The following lemma gives an equivalence to the existence of a unique polymer involution.

\begin{lemma}\label{lemma TFAE}
Let $O\subset (0,\infty)^2$, $O_3\subset (0,\infty)$, $h:O_3\rightarrow (0,\infty)$, and  $T_1^h, T_2^h$ be as in \eqref{eq-T1T2}. Then the following are equivalent:

\begin{enumerate}
\item $(T_1^h,\,T_2^h)(O\times O_3)=O$ and for every  $(r_1,r_2)\in O$ the function $O_3\owns y\mapsto T_2^h(r_1,r_2,y)=y\frac{r_2}{r_1}+h(y)$ is injective.
\item $G(s,y):=\bigl(y+\frac{h(y)}{s},\,ys+h(y)\bigr)$ is a bijection between $Q(O)\times O_3$ and $O$.
\item There exists a unique polymer involution $T^h$ on $O\times O_3$ adapted to $h$.  Moreover, 
\begin{equation}
T^h=(G\otimes \text{id})\circ\psi_{2,3}\circ (G\otimes \text{id})^{-1}, \label{T=G tensor id...}
\end{equation}
where $\psi_{2,3}(a,b,c)=(a,c,b)$ and $(G\otimes \text{id})(a,b,c):=(G(a,b),c)$.
\item There exists a polymer involution on $O\times O_3$ adapted to $h$ such that $T_3^h$ has no $y$-dependence.
\end{enumerate}

\end{lemma}

\begin{proof}
$(a)\Rightarrow (b)$:  Note that
\begin{equation}\label{equation-G-T1T2}
G(\frac{r_2}{r_1},y)=(T_1^h,T_2^h)(r_1,r_2,y)
\end{equation} implies $G(Q(O)\times O_3)=O$.  Injectivity of $G$ follows from  $\frac{\pi_2\circ G(s,y)}{\pi_1\circ G(s,y)}=s$ and the injectivity condition on $T_2^h$.

$(b)\Rightarrow (c)$:  We first show uniqueness.  Suppose $T=(T_1^h,T_2^h,T_3^h)$ is a polymer involution on $O\times O_3$ adapted to $h$.  For fixed $(r_1,r_2,y)\in O\times O_3$, with notation as in \eqref{tildes},  we have $T(\t{r}_1,\t{r}_2,\t{y})=(r_1,r_2,y)$ since $T$ is an involution.   Using \eqref{persistence of ratios} we have
\[
(r_1,r_2)=(T_1^h,T_2^h)(\t{r}_1,\t{r}_2,\t{y})=G(\frac{r_2}{r_1},\t{y}).
\]
Therefore 
\begin{equation}
G^{-1}(r_1,r_2)=(\frac{r_2}{r_1},\,T_3^h(r_1,r_2,y)).\label{equation - first g inverse}
\end{equation}
Since $G^{-1}$ has no $y$-dependence, neither does $T_3^h$.  One can now check that
\begin{equation}
T=(G\otimes \text{id}) \circ \psi_{2,3} \circ (G\otimes \text{id})^{-1} \label{T3-form}
\end{equation}
proving uniqueness.  Existence follows by simply setting 
$T_3^h(r_1,r_2,y)=\pi_2\circ G^{-1}(r_1,r_2)$. This forces equality \eqref{T3-form}, the right side of which is indeed a polymer involution adapted to $h$. 

$(c)\Rightarrow (d)$ is clear.

$(d)\Rightarrow (a)$:  Let $T$ be a polymer involution on $O\times O_3$ adapted to $h$ for which $T_3^h$ has no $y$-dependence.  Clearly the first two components of $T$, $(T_1^h,T_2^h)$, surject onto $O$.   Now fix $(r_1,r_2)\in O$.   Since $T_1^h(r_1,r_2,y)=\frac{r_1}{r_2}T_2^h(r_1,r_2,y)$ and $T$ is itself injective, we have injectivity of $y\mapsto T_2^h(r_1,r_2,y)$.  
\end{proof}

\begin{remark}\label{remark after TFAE lemma}
Note that the conditions in part (a) of Lemma \ref{lemma TFAE} depend only on the sets $O,\, O_3$, and the function $h$. The condition $(T^h_1,T^h_2)(O\times O_3)=O$ in part (a) is equivalent to the condition that the mapping $O\times O_3\owns (r_1,r_2,y)\mapsto T^{h,y}(r_1,r_2)$ surjects onto $O$ (recall definition \eqref{equation - Thy}). 
\end{remark}

When the polymer involution $T$ is such that $T_3^h$ has no $y$-dependence, we will simply write $T_3^h(r_1,r_2)$.
% As an example, if $O_3$ is connected and $h$ is differentiable, then the injectivity condition is satisfied if either
% \begin{align*}
% h'(y) < -c^+ \qquad \text{or} \qquad h'(y) > -c^-
% \end{align*}
% for all $y\in O_3$, where $c^+ = \sup\{s \in O_1/O_2\}$ and $c^- = \inf\{s\in O_1/O_2\}$. 
The following lemma gives conditions for when $T^{h,Y}$-invariance is equivalent to $T^h$-invariance.
\begin{lemma}\label{lemma distr. TFAE}
Suppose $O,\, O_3$, and $h$ satisfy one of the equivalent conditions in Lemma \ref{lemma TFAE}.  Let $(R^1,R^2,Y)$ be a random vector taking values in $O\times O_3$ and assume that $Y$ is independent of the pair $(R^1,R^2)$. Let $T^h$ be the unique polymer involution adapted to $h$, defined by \eqref{T=G tensor id...}, and write $\t{Y}=T_3^h(R^1,R^2)$.  Then the following are equivalent:

\begin{enumerate}
\item $(R^1,R^2)$ is $T^{h,Y}$-invariant.
\item $R^2/R^1$ is independent of $\t{Y}$ and $\t{Y} \overset{d}{=} Y$.
\item $(R^1,R^2,Y)$ is $T^h$-invariant.
\end{enumerate}
\end{lemma}

\begin{proof}
$(a) \Leftrightarrow (b)$:   Put $(\t{R}^1,\t{R}^2)=(T_1^h,T_2^h)(R^1,R^2,Y)$. Using equations \eqref{equation-G-T1T2} and \eqref{equation - first g inverse}, 
\begin{align*}
&\,\,\,\,\,  G(R^2/R^1,Y)=(\t{R}^1,\t{R}^2)\overset{d}{=}(R^1,R^2)\\
&\Leftrightarrow(R^2/R^1,Y)\overset{d}{=} G^{-1}(R^1,R^2)=(R^2/R^1,\t{Y}) \\
&\Leftrightarrow R^2/R^1 \text{ is independent of } \t{Y}  \text{ and } Y\overset{d}{=} \t{Y}.
\end{align*}
$(c) \Rightarrow (a)$ is clear.  We now show that $(a)$ and $(b)$ imply $(c)$. Since $T_3^h$ has no $y$-dependence, $Y$ is independent of the pair $(R^2/R^1,\t{Y})$.  Therefore the triple $(R^2/R^1,Y,\t{Y})$ is independent.  Thus $(\t{R}^1,\t{R}^2)=G(R^2/R^1,Y)$ is independent of $\t{Y}$. Now combining $(a)$ and $\t{Y}\overset{d}{=} Y$ we get $(\t{R}^1,\t{R}^2,\t{Y})\overset{d}{=} (R^1,R^2,Y)$.
\end{proof}

We now give an analogue of Lemma \ref{lemma TFAE} in which $h$ and $T^h$ are continuously differentiable. Given a differentiable transformation $F:U\rightarrow \R^m$, where $U\subset \R^n$ is open, 
%write $F= (F_1, \ldots, F_m)$ and 
use the notations $DF(u)$ and $D[F](u)$ to denote the Jacobian matrix of $F$ evaluated at the point $u\in U$.   
% Explicitly,
% \begin{align*} DF(u):= \left[ \frac{\p F_i}{\p x_j}(u)\right]_{1\leq i\leq m,1\leq j \leq n}.
% \end{align*}
When $m=n$ we say $F$ is a $C^1$-diffeomorphism if $F$ is injective, continuously differentiable, and its Jacobian matrix is invertible throughout $U$.

\begin{lemma}\label{lemma differentiable TFAE}
Let $O\subset (0,\infty)^2$, $O_3\subset (0,\infty)$, $h:O_3\rightarrow (0,\infty)$, and  $T_1^h, T_2^h$ be as in \eqref{eq-T1T2}. Further assume $O$ and $O_3$ are open, $O_3$ is connected, and $h$ is continuously differentiable.  Then the following are equivalent:

\begin{enumerate}
\item $(T_1^h,T_2^h)(O\times O_3) = O$ and the following function does not vanish on $Q(O)\times O_3$ 
\begin{equation} \label{definition-L} L(s,y) :=s + h'(y).\end{equation} 
\item $G(s,y):=(y+\frac{h(y)}{s},\,ys+h(y))$ is a $C^1$-diffeomorphism between $Q(O)\times O_3$ and $O$.  Moreover its  Jacobian matrix and determinant are given by 
\begin{align}
DG(s,y)&= \begin{bmatrix}
-h(y)/s^2 & L(s,y)/s\\
y & L(s,y)
\end{bmatrix}, & \det DG(s,y)&=-\frac{L(s,y)}{s}\left(y+\frac{h(y)}{s}\right). \label{equation Matrix DG} 
\end{align}

\item There exists a unique $C^1$-diffeomorphic polymer involution $T^h$ on $O\times O_3$ adapted to $h$.  Moreover $T_3^h$ has no $y$ dependence and the Jacobian matrix and determinant of $T^h$ are given by
\begin{equation}\label{equation-DT}
D T^h(r_1,r_2,y) = \frac{1}{r_1}\begin{bmatrix}
h(y)/s & -h(y)/s^2 & L(s,y)r_1/s \\
-ys & y &  L(s,y)r_1\\
 \t{y}s/L(s,\t{y}) &  h(\t{y})/\big(s L(s,\t{y})\big) & 0
\end{bmatrix},
\end{equation} 

\[\det DT^h(r_1,r_2,y)=-\left(\frac{y}{r_1}+\frac{h(y)}{r_2}\right)\frac{L(s,y)}{L(s,\t{y})}, \]
 where $s=\frac{r_2}{r_1}$ and $\t{y} = T_3^h(r_1,r_2)$.
\item There exists a differentiable polymer involution on $O\times O_3$ adapted to $h$.
\end{enumerate}
\end{lemma}

\begin{proof}
$(a)\Rightarrow (b)$: For fixed $(r_1,r_2)\in O$, since $y\mapsto \frac{\p T_2^{h}}{\p y}(r_1,r_2,y)= L(\frac{r_2}{r_1},y)$ does not vanish on the connected set $O_3$, the conditions of Lemma \ref{lemma TFAE}-(a) are satisfied.  Therefore $G$ is a bijection.  The continuous differentiability of $h$ now implies that $G$ is continuously differentiable.  The Jacobian matrix and determinant of $G$ can now be calculated.  Notice that for all $(s,y)\in Q(O)\times O_3$, $y+h(y)/s=\pi_1\circ G(s,y)\in \pi_1(O)\subset (0,\infty)$.  Thus the Jacobian determinant of $G$ does not vanish on $Q(O)\times O_3$, which shows it is a $C^1$-diffeomorphism.

$(b)\Rightarrow (c)$: Since $G$ is a bijection, Lemma \ref{lemma TFAE} gives existence and uniqueness of the polymer involution $T^h=(G\otimes \text{id})\circ \psi_{2,3}\circ (G\otimes \text{id})^{-1}$.  Since $G$ is a $C^1$-diffeomorphism, the inverse function theorem tells us $T^h$ is a $C^1$-diffeomorphism as well. Now fix $(r_1,r_2,y)\in O\times O_3$ and put $(s,\t{y})=\big(\frac{r_2}{r_1}, \ T_3^h(r_1,r_2)\big)$.    By  \eqref{equation - first g inverse}
\begin{equation}
(s,\t{y})=G^{-1}(r_1,r_2).\label{equation- g inverse}
\end{equation}
   $DG^{-1}(r_1,r_2)$ is now the inverse of the matrix $DG(G^{-1}(r_1,r_2))=DG(s,\t{y})$.  \eqref{equation- g inverse} implies $(r_1,r_2)=G(s,\t{y})=\big(\t{y}+h(\t{y})/s,\,\t{y}s+h(\t{y})\big)$.  Using this one can show that
\begin{align}
DG^{-1}(r_1,r_2)&= \frac{1}{r_1}\begin{bmatrix}
-s & 1 \\
 \frac{s\t{y}}{L(s,\t{y})} & \frac{h(\t{y})}{s L(s,\t{y})}
\end{bmatrix} & &\text{ and } & \det DG^{-1}(r_1,r_2)=-\frac{s}{r_1 L(s,\t{y})}.\label{equation DG inverse calc}
\end{align}
Using equations \eqref{equation Matrix DG}, \eqref{equation- g inverse}, and \eqref{equation DG inverse calc} we can compute
\begin{align*}
DT^h(r_1, r_2, y) = & \left[D(G\otimes id)\left(\psi_{2,3} \circ (G^{-1}\otimes id)(r_1, r_2, y)\right)\right] 
\cdot \left[D\psi_{2,3}\left((G^{-1}\otimes id)(r_1, r_2, y)\right)\right] \\
&\cdot \left[D(G^{-1}\otimes id)(r_1, r_2, y)\right]\\
= & \left[DG(s,y)\otimes 1\right]
\cdot\left[ D\psi_{2,3}\right]
\cdot\left[ DG^{-1}(r_1,r_2) \otimes 1\right]\\
= &\begin{bmatrix}
\frac{-h(y)}{s^2} & \frac{L(s,y)}{s} & 0\\
y & L(s,y) & 0\\
0 & 0 & 1
\end{bmatrix}\cdot \begin{bmatrix}
1 & 0 & 0\\
0 & 0 & 1\\
0 & 1 & 0
\end{bmatrix}\cdot \frac{1}{r_1}\begin{bmatrix}
-s & 1 & 0\\
\frac{s \t{y}}{L(s,\t{y})} & \frac{h(\t{y})}{s L(s,\t{y})} & 0\\
0 & 0 & r_1
\end{bmatrix}\\
= &\frac{1}{r_1}\begin{bmatrix}
h(y)/s & -h(y)/s^2 & L(s,y)r_1/s \\
-ys & y &  L(s,y)r_1\\
 \t{y}s/L(s,\t{y}) &  h(\t{y})/\big(s L(s,\t{y})\big) & 0
\end{bmatrix}
\end{align*}
and
\begin{align*}
\det(DT^h(r_1,r_2,y)) &= \det(DG(s,y)) \det(D\psi_{2,3}) \det(DG^{-1}(r_1,r_2))\\
&=-\frac{L(s,y)}{s}\left(y+\frac{h(y)}{s}\right)(-1) \left(-\frac{s}{r_1L(s,\t{y})}\right)\\
&= -\left(\frac{y}{r_1}+\frac{h(y)}{r_2}\right)\frac{L(s,y)}{L(s,\t{y})}.
\end{align*}

$(c)\Rightarrow (d)$ is clear.

$(d)\Rightarrow (a)$:  If $T$ is a differentiable polymer involution adapted to $h$, then its Jacobian matrix has the same entries as the $2\times 3$ upper portion of \eqref{equation-DT}, as $(T_1^h,T_2^h)$ are completely determined.  Therefore the determinant of the top-left $2\times 2$ minor of the Jacobian matrix of $T$ is zero. Thus $L$ vanishing at a point $(s,y)\in Q(O)\times O_3$ would imply the Jacobian determinant of $T$ vanishes at any point $(r_1,r_2,y)\in O\times O_3$ such that $\frac{r_2}{r_1}=s$.  Since $T\circ T$ is the identity function, the Jacobian determinant of $T$ cannot vanish on $O\times O_3$.  Thus $L$ cannot vanish on $Q(O)\times O_3$.
\end{proof}

\section{Proof of Theorem \ref{thm-h-linear}}\label{sec-proof of second thm}

We begin by using Lemma \ref{lemma differentiable TFAE} to give another useful equivalence to $T$-invariance under some regularity assumptions. In the appendix of \cite{T2016}, Thiery uses a specific case of the following proposition to prove the invariance of the inverse-beta model.  It can also be used to prove invariance of the other three basic beta-gamma models.

\begin{proposition}\label{prop-q-condition}
Let $(R^1,R^2,Y)$ be a random vector with density $\rho$ and assume $Y$ is independent of $(R^1,R^2)$.  Suppose the support of $\rho$ equals $O\times O_3$ where $O\subset (0,\infty)^2$ is open and $O_3\subset (0,\infty)$ is open and connected.  Let $h:O_3\rightarrow (0,\infty)$ be continuously differentiable and $T$ be a differentiable polymer involution adapted to $h$ on $O\times O_3$. Then $(R^1,R^2,Y)$ is $T$-invariant if and only if

\[
 q\circ T(x)=q(x) \qquad \text{for a.e.\ } x\in O\times O_3
\]
where $q(r_1,r_2,y):=\frac{r_2}{|L(r_2/r_1,y)|}\rho(r_1,r_2,y)$ and $L(s, y)= s + h'(y)$, as given in \eqref{definition-L}.
\end{proposition}

\begin{proof}[Proof of Proposition \ref{prop-q-condition}]
Recall the notation \eqref{tildes}. By Lemma \ref{lemma differentiable TFAE}, $L$ does not vanish on $Q(O)\times O_3$ and $T$ is in fact a $C^1$-diffeomorphism with 
\[
\det DT(r_1,r_2,y)=-\left(\frac{y}{r_1}+\frac{h(y)}{r_2}\right)\frac{L(r_2/r_1, y)}{L(r_2/r_1, \t{y})}=-\frac{\t{r}_2 L(r_2/r_1,y)}{r_2 L(r_2/r_1,\t{y})}.
\]
Therefore $T(R^1,R^2,Y)$ has density 
\begin{align*}
\widehat{\rho}(x)&:=\rho(T^{-1}(x))\left|\det DT^{-1}(x)\right| = \rho(T(x))\left|\det DT(x)\right|
\end{align*}
supported on $x\in O\times O_3$. Thus $T$-invariance of $(R^1,R^2,Y)$ is equivalent to $\rho(x)=\widehat{\rho}(x)$ a.e.\ on $O\times O_3$.

Using \eqref{persistence of ratios} we can explicitly write $\widehat{\rho}(r_1,r_2,y)=\rho(\t{r}_1,\t{r}_2,\t{y})\left|\frac{\t{r}_2 L(r_2/r_1,y)}{r_2 L(\t{r}_2/\t{r}_1,\t{y})}
\right|$. After rearranging terms, the condition $\rho(x)=\widehat{\rho}(x)$ for a.e.\ $x\in O\times O_3$ yields the desired result. 
\end{proof}

We now prove the first main result.

\begin{proof}[Proof of Theorem \ref{thm-h-linear}]
By assumption $\ID$ has density $\rho(r_1,r_2,y)=f_1(r_1)f_2(r_2)f_3(y)$. By Lemma \ref{lemma differentiable TFAE}, there exists a unique differentiable polymer involution $T^h$ on $\Os$ adapted to $h$ and the function $L(s,y)=s+h'(y)$ does not vanish on the set $\frac{O_2}{O_1}\times O_3$. Applying Proposition \ref{prop-q-condition} gives $q\circ T^h = q$ a.e.\ on $\Os$. Since $f_1$, $f_2$, $f_3$, and $T^h$ are continuous, this equality holds everywhere on $\Os$.  Since the support of $f_j$ equals $O_j$, we can further assume $f_j(x) = \exp(\eta_j(x))$ for $x\in O_j$, $j=1,2,3$. Note that $\eta_j$ has the same differentiability properties as $f_j$. Set $s=\frac{r_2}{r_1}$ and recall the notation \eqref{tildes}. Taking logarithms of the equality $q\circ T^h = q$ then computing the total derivative we obtain
\begin{align}\label{equation-DqT=Dq}
D[\log q](\t{r}_1, \t{r}_2,\t{y}) \cdot DT^h(r_1, r_2,y) = D[\log q](r_1, r_2,y),
\end{align}
where $DT^h$ is given in \eqref{equation-DT} and 
\begin{align*}
D[\log q](r_1, r_2, y) = \begin{bmatrix}
\frac{r_2}{r_1^2 L(s, y)} + \eta_1'(r_1),& \frac{h'(y)}{r_2 L(s,y)} + \eta_2'(r_2), & -\frac{h''(y)}{L(s,y)} + \eta_3'(y)
\end{bmatrix}.
\end{align*}
Using the fact that $T^h$ is an involution and \eqref{persistence of ratios}, $r_2 = T_2^h(\t{r}_1, \t{r}_2, \t{y})= \t{y} s + h(\t{y})$. One can then check that
\[
DT^h(r_1, r_2, y) \cdot [
r_1,\,
r_2,\,
0
]^T=[
0,\,
0,\,
r_2/L(s,\t{y})
]^T.
\]
Thus multiplying both sides of equation \eqref{equation-DqT=Dq} on the right by $[r_1,\,r_2,\,0]^T$ gives
\begin{align}
 1 + r_1\eta_1'(r_1)+ r_2\eta_2'(r_2)  = r_2 g(s,\t{y}), \label{r1 r2 split}
\end{align}
where 
\[g(s,y):=\frac{\eta_3'(y)}{L(s,y)}-\frac{h''(y)}{L(s,y)^2}
%,\qquad L_1(s,y):=s+h'(y), \qquad \text{ and } \,\, s=\frac{r_2}{r_1}
.\]
% Note that $L_1(s, y) = L(r_1,r_2,y)$. The functions $g$ and $L_1$ are defined on the open set $\frac{O_2}{O_1}\times O_3$.  Then by assumption $L_1$ does not vanish on this set.  

Applying the operator $\frac{\p^2}{\p r_1\p r_2}$ to the left-hand side of \eqref{r1 r2 split} gives zero.  We now exploit the fact that  $\frac{\p^2}{\p r_1\p r_2}$ applied to the right hand side must equal zero to ultimately arrive at the conclusion that $h''(y)=0$.  

Note that if $f$ is differentiable then for all non-negative integers $k$ and $n$,
\begin{align}\label{equation-derivatives}
D\left[ \frac{s^k f(y)}{L(s,y)^n}\right] (s,y) &= s^{k-1}\begin{bmatrix}
\frac{1}{L(s,y)^n}, & \frac{-n f(y)}{L(s,y)^{n+1}}
\end{bmatrix}\cdot \begin{bmatrix}
kf(y) & s f'(y)\\
s & s h''(y)
\end{bmatrix}.
\end{align}
First calculate, using \eqref{equation-DT} and \eqref{equation-derivatives},
\begin{align*}
\frac{\p}{\p r_1} (r_2 g(s,\t{y})) &= r_2 Dg(s,\t{y}) \cdot \left[
\frac{\p s}{\p r_1},\,
\frac{\p \t{y}}{\p r_1}
\right]^T\\
&= s^2 Dg(s,\t{y}) \cdot \left[
-1,\,
\frac{\t{y}}{L(s,\t{y})}
\right]^T\\
&=  s\left[\frac{1}{L(s,\t{y})},\, -\frac{\eta_3'(\t{y})}{L(s,\t{y})^2}\right]\cdot \begin{bmatrix}
0 & s \eta_3''(\t{y})\\
s & sh''(\t{y})
\end{bmatrix}\cdot \left[
-1,\,
\frac{\t{y}}{L(s,\t{y})}
\right]^T\\
&\qquad - s\left[\frac{1}{L(s,\t{y})^2},\, -\frac{2h'(\t{y})}{L(s,\t{y})^3}\right]\cdot \begin{bmatrix}
0 & s h''(\t{y})\\
s & sh''(\t{y})
\end{bmatrix}\cdot \left[
-1,\,
\frac{\t{y}}{L(s,\t{y})}
\right]^T\\
&= t(s,\t{y}) := \sum_{j=2}^4 \frac{s^2 \kappa_j(\t{y})}{L(s,\t{y})^j},
\end{align*}
where
\begin{align*}
\kappa_2(y) &= y \eta_3''(y) + \eta_3'(y)\\
\kappa_3(y) &=  - y h''(y) \eta_3'(y)- y h'''(y) - 2h''(y)\\
\kappa_4(y) &= 2y h''(y)^2.
\end{align*}
Taking an $r_2$ partial derivative and multiplying by $r_1$, by \eqref{equation-DT}
\begin{align*}
0 &= r_1\frac{\p^2}{\p r_2 \p r_1} (r_2 g(s,\t{y}))  = r_1\frac{\p}{\p r_2} t(s,\t{y})\\
&=r_1 Dt(s,\t{y}) \cdot\left[
\frac{\p s}{\p r_2},\,
\frac{\p \t{y}}{\p r_2}
\right]^T\\
&=Dt(s,\t{y}) \cdot \left[
1,\,
\frac{h(\t{y})}{sL(s,\t{y})}
\right]^T.
\end{align*}
This equality holds for all $(r_1,r_2,y)\in \Os$.  Since $T^h$ is an involution on $\Os$, it also holds after interchanging $(r_1,r_2,y)\leftrightarrow (\t{r}_1,\t{r}_2,\t{y})$.  Notice that, by \eqref{persistence of ratios}, $s=\frac{r_2}{r_1}$ is unaffected by this interchange.  Therefore, applying this interchange and using \eqref{equation-derivatives} 
\begin{align*}
0=Dt(s,y) \cdot \left[
1,\,
\frac{h(y)}{sL(s,y)}
\right]^T = \sum_{j=2}^4 s\begin{bmatrix}
\frac{1}{L(s,y)^j}, & \frac{-j \kappa_j(y)}{L(s,y)^{j+1}}
\end{bmatrix}\cdot\begin{bmatrix}
2 \kappa_j(y) & s \kappa_j'(y)\\
s & s h''(y)
\end{bmatrix}\cdot \left[
1,\,
\frac{h(y)}{sL(s,y)}
\right]^T
\end{align*}
for all $(s,y)\in \frac{O_2}{O_1}\times O_3$.
Multiplying by $L(s,y)^6/s$ gives
\begin{align}\label{equation-whyisitpolynomial}
0 &= \sum_{j=2}^4 \begin{bmatrix}
L(s,y)^{5-j}, & -j \kappa_j(y)L(s,y)^{4-j}
\end{bmatrix}\cdot\begin{bmatrix}
2 \kappa_j(y) &  \kappa_j'(y)\\
s &  h''(y)
\end{bmatrix}\cdot \left[
L(s,y),\,
h(y)
\right]^T.
\end{align}

Now fix $y\in O_3$.  The right hand side is now a fourth degree polynomial in $s$ which vanishes on the open set $\frac{O_2}{O_1}$.  It must therefore vanish at all values $s\in \R$.  Taking $s=-h'(y)$ so that $L(s,y) =0$, \eqref{equation-whyisitpolynomial} gives
\[0 = -4\kappa_4(y)h(y)h''(y) = -8 y h(y) h''(y)^3.\]
The fact that $y$ and $h(y)$ are positive implies $h''(y) = 0$. Since this holds for all $y\in O_3$, which we assumed to be connected, $h$ has the form $h(y) = a+by$ where $a,b$ are real numbers. The condition $a\vee b>0$ follows from the fact that $h$ maps a subset of $(0,\infty)$ into $(0,\infty)$.  
\end{proof}

\section{Reflection and scaling}\label{sec-modifications}
We describe two procedures which preserve $T$-invariance. By applying these procedures to the four basic beta-gamma models, we can obtain a $T$-invariant model corresponding to $h(y)=a+by$ for each choice of $a,b$ such that $a\vee b>0$.

We first define the reflection procedure. Let $T$ be a polymer involution adapted to $h$ on $\Os$ and assume that $h$ is injective so that $h:O_3\to h(O_3)$ is a bijection. Define the mapping $\rho(r_1,r_2,y) := (r_2, r_1, h(y))$. Define the mapping and the random vector \begin{equation}\label{equation T hat}
\widehat{T} := \rho\circ T \circ \rho^{-1}\quad \text{ and } \quad \bigl(\widehat{R}^1,\widehat{R}^2,\widehat{Y}\bigr):=\flipID.
\end{equation} One can then check that $\widehat{T}$ is a polymer involution adapted to $h^{-1}$ on $O_2\times O_1\times h(O_3)$. Furthermore, $(R^1,R^2,Y)$ is $T$-invariant with respect to $h$ if and only if $\bigl(\widehat{R}^1,\widehat{R}^2,\widehat{Y}\bigr)$ is $\widehat{T}$-invariant with respect to $h^{-1}$.

In the directed polymer setting, this procedure of mapping 
\begin{align*}
h &\mapsto h^{-1} &
\ID & \mapsto \bigl(\widehat{R}^1,\widehat{R}^2,\widehat{Y}\bigr) &
T&\mapsto \widehat{T}
%\label{reflection-procedure}
\end{align*}
corresponds to interchanging the horizontal and vertical coordinates while remaining in the same framework.  This is illustrated in Figure \ref{fig-reflection}.

\begin{figure}[ht]
  \centering
  \begin{tikzpicture}[scale=.7]
    
   \draw [help lines] (0,0) grid (5,5);	  
   \draw [thick, <->] (0,5) -- (0,0) -- (5,0);
   \draw [thick] (2,3) -- (3,3) -- (3,2);
   \node [left,scale=.9] at (0,2.5) {$R^2$};
   \node [below,scale=.9] at (2.5,0) {$R^1$};
   \node [above,scale=.9] at (2.5,3) {$Y$};
   \node [right,scale=.9] at (3,2.5) {$h(Y)$};
   \draw[fill] (3,3) circle [radius=.05];
   \node [below] at (2.5, -1) {Original};
   %\node [above right] at (3.5,3.5) {$x$};

   \begin{scope}[shift={(7.5,0)}]	   
	   \draw [help lines] (0,0) grid (5,5);	  
	   \draw [thick, <->] (0,5) -- (0,0) -- (5,0);
	   \draw [thick] (2,3) -- (3,3) -- (3,2);
	   \node [left,scale=.9] at (0,2.5) {$R^1$};
	   \node [below,scale=.9] at (2.5,0) {$R^2$};
	   \node [above,scale=.9] at (2.5,3) {$h(Y)$};
	   \node [right,scale=.9] at (3,2.5) {$Y$};
	   \draw[fill] (3,3) circle [radius=.05]; 
   	   \node [below] at (2.5, -1) {Reflected};
   \end{scope}
   
   \begin{scope}[shift={(15,0)}]	   
	   \draw [help lines] (0,0) grid (5,5);	  
	   \draw [thick, <->] (0,5) -- (0,0) -- (5,0);
	   \draw [thick] (2,3) -- (3,3) -- (3,2);
	   \node [left,scale=.9] at (0,2.5) {$\widehat{R}^2$};
	   \node [below,scale=.9] at (2.5,0) {$\widehat{R}^1$};
	   \node [above,scale=.9] at (2.5,3) {$\widehat{Y}$};
	   \node [right,scale=.9] at (3,2.5) {$h^{-1}(\widehat{Y})$};
	   \draw[fill] (3,3) circle [radius=.05]; 
	   \node [below] at (2.5, -1) {$\flipID=\bigl(\widehat{R}^1,\widehat{R}^2,\widehat{Y}\bigr)$};
   \end{scope}

  \end{tikzpicture}
    \caption{Reflection}
  \label{fig-reflection}
\end{figure}
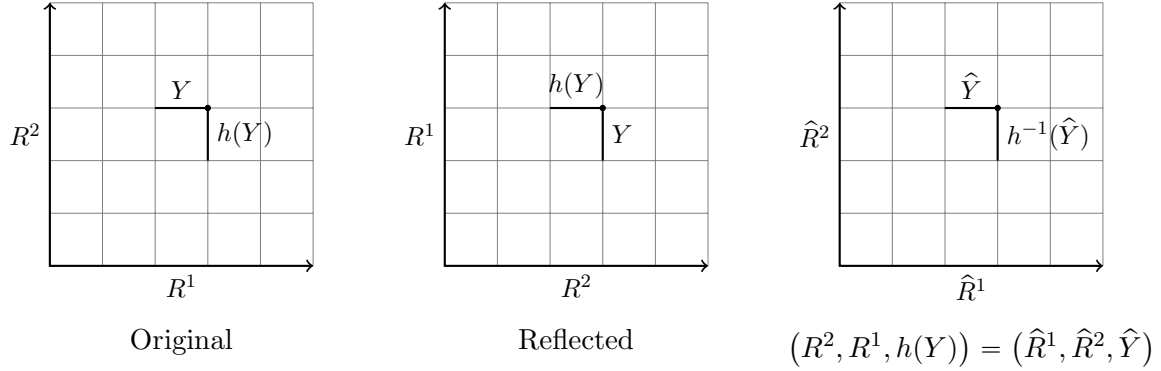

We now define the scaling procedure. If $O\subset (0,\infty)$ and $c$ is a positive constant, define $cO:=\{cx: x\in O\}$. Note that $cO\subset (0,\infty)$.
Let $c_1,c_2$ be positive constants. Let $T$ be a polymer involution adapted to $h$ on $\Os$. Define the mapping $\sigma(r_1,r_2,y) := (c_1r_1,c_2r_2,c_1y)$. Define the two mappings and the random vector 
\begin{align} \label{equation Tcheck hcheck}
\widecheck{T}&:=\sigma\circ T \circ \sigma^{-1}& \widecheck{h}(y) &:= c_2 h(\tfrac{y}{c_1})& (\widecheck{R}^1,\widecheck{R}^2, \widecheck{Y})&:=\scaleID.
\end{align}
One can check that $\widecheck{T}$ is a polymer involution adapted to $\widecheck{h}$ on $\scaleOs$. Furthermore, $\ID$ is $T$-invariant with respect to $h$ if and only if $(\widecheck{R}^1,\widecheck{R}^2, \widecheck{Y})$ is $\widecheck{T}$-invariant with respect to $\widecheck{h}$.

In the directed polymer setting, this procedure of mapping

\begin{align*}
h & \mapsto \widecheck{h} &
\ID & \mapsto (\widecheck{R}^1,\widecheck{R}^2, \widecheck{Y}) &
T & \mapsto \widecheck{T}%\label{scaling procedure}
\end{align*}
corresponds to scaling the horizontal axis weights by $c_1$ and the vertical axis weights by $c_2$ while remaining in the same framework. This procedure is illustrated in Figure \ref{fig-scaling}.

\begin{figure}[ht]
  \centering
  \begin{tikzpicture}[scale=.7]
    
   \draw [help lines] (0,0) grid (5,5);	  
   \draw [thick, <->] (0,5) -- (0,0) -- (5,0);
   \draw [thick] (2,3) -- (3,3) -- (3,2);
   \node [left,scale=.9] at (0,2.5) {$R^2$};
   \node [below,scale=.9] at (2.5,0) {$R^1$};
   \node [above,scale=.9] at (2.5,3) {$Y$};
   \node [right,scale=.9] at (3,2.5) {$h(Y)$};
   \draw[fill] (3,3) circle [radius=.05]; 
   \node [below] at (2.5, -1) {Original};
   %\node [above right] at (3.5,3.5) {$x$};

   \begin{scope}[shift={(7.5,0)}]	   
	   \draw [help lines] (0,0) grid (5,5);	  
	   \draw [thick, <->] (0,5) -- (0,0) -- (5,0);
	   \draw [thick] (2,3) -- (3,3) -- (3,2);
	   \node [left,scale=.9] at (0,2.5) {$c_2R^2$};
	   \node [below,scale=.9] at (2.5,0) {$c_1R^1$};
	   \node [above,scale=.9] at (2.5,3) {$c_1Y$};
	   \node [right,scale=.9] at (3,2.5) {$c_2h(Y)$};
	   \draw[fill] (3,3) circle [radius=.05]; 
   \node [below] at (2.5, -1) {Scaled};
   \end{scope}
   
   \begin{scope}[shift={(15,0)}]	   
	   \draw [help lines] (0,0) grid (5,5);	  
	   \draw [thick, <->] (0,5) -- (0,0) -- (5,0);
	   \draw [thick] (2,3) -- (3,3) -- (3,2);
	   \node [left,scale=.9] at (0,2.5) {$\widecheck{R}^1$};
	   \node [below,scale=.9] at (2.5,0) {$\widecheck{R}^2$};
	   \node [above,scale=.9] at (2.5,3) {$\widecheck{Y}$};
	   \node [right,scale=.9] at (3,2.5) {$\widecheck{h}(\widecheck{Y})$};
	   \draw[fill] (3,3) circle [radius=.05]; 
	   \node [below] at (2.5, -1) {$\scaleID=(\widecheck{R}^1,\widecheck{R}^2,\widecheck{Y})$};
	   %\node [below] at (2.5, -2) {$\widetilde{h}(x) = c_2 h(\tfrac{x}{c_1})$};
   \end{scope}

  \end{tikzpicture}
    \caption{Scaling}
  \label{fig-scaling}
\end{figure}

One can also check that the reflection and scaling procedures commute.
By using the reflection and scaling procedures, the following lemma reduces the existence and uniqueness of $T$-invariant models corresponding to $h(y)=a+by$ where $a\vee b>0$ to the existence and uniqueness for values $(a,b)$ = $(0,1),(1,0),(1,-1)$, and $(-1,1)$.

For real numbers $a,b$ such that $a\vee b>0$, define
\begin{equation}
T^{(a,b)}(r_1,r_2,y):=\Bigl(y+(a+by)\frac{r_1}{r_2},\,y\frac{r_2}{r_1} + (a+by),\,\frac{r_1(r_2-a)}{r_2+br_1}\Bigr).\label{T-ab}
\end{equation}
One can check that when $h(y)=a+by$, \eqref{T=G tensor id...} implies that $T^{h}=T^{(a,b)}$.  The domain of $T^{(a,b)}$ is discussed prior to Lemma \ref{lemma-involution-properties}.

\begin{lemma}\label{lem-scaling}
Let $a,b$ be real numbers satisfying $a\vee b>0$, $h(y)=a+by$, and $T=T^{(a,b)}$ as defined in \eqref{T-ab}. Let $R^1, R^2,$ and $Y$ be random variables.
\begin{enumerate}
\item If $a=0$ and $b>0$, then $(R^1, R^2, Y)$ is $T$-invariant with respect to h if and only if $\left(R^1, \frac{1}{b}R^2, Y\right)$ is $T^{(0,1)}$-invariant with respect to $\widecheck{h}(y)=y$. %\eqref{model-IG}.
\item If  $a>0$ and $b=0$, then $(R^1, R^2, Y)$ is $T$-invariant with respect to $h$ if and only if $\left(R^1, \frac{1}{a}R^2, Y \right)$ is $T^{(1,0)}$-invariant with respect to $\widecheck{h}(y)=1$ %\eqref{model-G}.
\item If $a>0$ and $b<0$, then $(R^1, R^2, Y)$ is $T$-invariant with respect to $h$ if and only if $\left(-\frac{b}{a}R^1, \frac{1}{a}R^2, -\frac{b}{a}Y\right)$ is $T^{(1,-1)}$-invariant with respect to $\widecheck{h}(y)=1-y$.%\eqref{model-B}.
\item If $a<0$ and $b>0$, then $(R^1, R^2, Y)$ is $T$-invariant with respect to $h$ if and only if $\left( -\frac{b}{a}R^1,-\frac{1}{a} R^2, -\frac{b}{a} Y \right)$ is $T^{(-1,1)}$-invariant with respect to $\widecheck{h}(y)=y-1$. %\eqref{model-IB}.
\item If $a,b>0$, then $(R^1, R^2, Y)$ is $T$-invariant with respect to $h$ if and only if $\left(\frac{b}{a}R^1, \frac{1}{a}R^2,  \frac{b}{a}Y\right)$ is $T^{(1,1)}$-invariant with respect to $\widecheck{h}(y)=y+1$. %\eqref{model-IB}.
\item If $a=1$ and $b=1$, then $\ID$ is $T$-invariant with respect to $h$ if and only if $(R^2,R^1,1+Y)$ is $T^{(-1,1)}$-invariant with respect to $h^{-1}(y) = y-1$.
\end{enumerate}
\end{lemma}

\begin{proof}
Let $c_1,c_2$ be positive constants.  After applying the scaling procedure with $(c_1,c_2)$, with notation as in \eqref{equation Tcheck hcheck}, one can check that \[
\widecheck{h}(y)=ac_2+\frac{bc_2}{c_1}y \quad\text{ and }\quad\widecheck{T}=T^{(ac_2,\frac{bc_2}{c_1})}.
\]  Recall that $(\widecheck{R}^1,\widecheck{R}^2,\widecheck{Y})=\scaleID$ is $\widecheck{T}$-invariant with respect to $\widecheck{h}$ if and only if $\ID$ is $T$-invariant with respect to $h$.  Now (a) through (e) follow by taking \[(c_1,c_2)=\left(1,\frac{1}{b}\right), \left(1,\frac{1}{a}\right), \left(-\frac{b}{a},\frac{1}{a}\right), \left(-\frac{b}{a},-\frac{1}{a}\right), \left(\frac{b}{a},\frac{1}{a}\right)\] respectively. 

For part (f), after applying the reflection procedure, with notation as in \eqref{equation T hat}, one can check that $\widehat{T}=T^{(-1,1)}$. Since $(\widehat{R}^1, \widehat{R}^2, \widehat{Y}) = (R^2, R^1, 1+Y)$ is $\widehat{T}$-invariant with respect to $h^{-1}(y)= y-1$ if and only if $(R^1, R^2, Y)$ is $T$-invariant with respect to $h(y)=y+1$, the result follows.
\end{proof}

\section{Proof of Theorem \ref{thm-classify-linear}}\label{sec-proof of first thm}

The following two theorems, due to Seshadri and Weso{\l}owski (2003) and Lukacs (1955) give characterizations of gamma and beta random variables, which will be used in the sequel.

\begin{theorem}[\cite{SW2003}]\label{thm-SeWe}
Let $A$ and $B$ be non-degenerate independent random variables taking values in $(0,1)$. Then the pair $(C,D):=\left(\frac{1-B}{1-AB},\, 1-AB\right)$ is independent if and only if there exist positive constants $p,q,r$ such that $(A,B)\sim \text{Be}(p,q)\otimes \text{Be}(p+q,r)$, in which case $(C,D)\sim \text{Be}(r,q)\otimes \text{Be}(r+q,p)$.
\end{theorem}

\begin{theorem}[\cite{L1955}]\label{thm-lukacs}
Let $A$ and $B$ be non-degenerate independent positive random variables. Then the pair $(C,D):=\left(A+B,\, \frac{A}{A+B}\right)$ is independent if and only if there exist positive constants $\lambda_A,\lambda_B,\beta$ such that $(A,B)\sim \text{Ga}(\lambda_A,\beta)\otimes \text{Ga}(\lambda_B,\beta)$, in which case $(C,D)\sim \text{Ga}(\lambda_A+\lambda_B,\beta)\otimes \text{Be}(\lambda_A,\lambda_B)$.
\end{theorem}

Notice that the mapping $(A,B)\mapsto \left(A+B,\, A/(A+B)\right)$ has the inverse $(A,B)\mapsto \left(AB,\, A(1-B)\right)$. The following statement is a corollary of Theorem \ref{thm-lukacs}.

\begin{corollary}\label{corollary-Lu}
Let $A$ and $B$ be non-degenerate independent random variables. Further assume that $A$ is positive and $B$ takes values in $(0,1)$. Then the pair $(C,D):= (AB,\, A(1-B))$ is independent if and only if there exist positive constants $\lambda_A,\lambda_B,\beta$ such that $(A,B)\sim \text{Ga}(\lambda_A+\lambda_B,\beta)\otimes \text{Be}(\lambda_A,\lambda_B)$ in which case $(C,D)\sim \text{Ga}(\lambda_A,\beta)\otimes \text{Ga}(\lambda_B,\beta)$.\\
\end{corollary}

The next lemma constrains the sets on which $T^{(a,b)}$ (as defined by \eqref{T-ab}) can be a polymer involution.  To specify this constraint, we define the following sets. For real numbers $(a,b)$ such that $a\vee b>0$, 
\begin{gather*}
V_a^{\pm}:=\{x>0: \pm (x-a)>0\},\qquad
W_{a,b}^{\pm}:=\{x>0: \pm (a+bx)>0\}\\
D_{a,b}^{\pm}:=W_{a,b}^{\pm}\times V_a^{\pm}\times W_{a,b}^{+}.
\end{gather*}

\begin{lemma}\label{lemma-involution-properties}
Let $a,b$ be real numbers satisfying $a\vee b>0$.  Let $O_j\subset (0,\infty)$ for $j=1,2,3$ such that $O_3$ is not a singleton.  If $T^{(a,b)}$, as defined in \eqref{T-ab}, is a polymer involution on $\Os$ with respect to $h$ of the form $h(y) = a+by$ then $\Os \subset D_{a,b}^{+}$ or $\Os \subset D_{a,b}^{-}$ assuming $D_{a,b}^{\pm}$ is non-empty.
\end{lemma}
\begin{proof} We first show the following holds:
\begin{itemize}
\item[(i)]  For all $(r_1,r_2)\in O_1\times O_2$, the three numbers $a+br_1, \frac{r_2}{r_1}+b, r_2-a$ are all either strictly positive, strictly negative, or equal to zero.
\end{itemize}
Fix $(r_1,r_2,y)\in \Os$ and put $\t{y}=T^{(a,b)}_3(r_1,r_2)=\frac{r_1(r_2-a)}{r_2+br_1}$. Then the following two equalities hold
\begin{gather}
r_2-a=\t{y}(\frac{r_2}{r_1}+b)\label{7-3}, \qquad
a+br_1=\frac{r_1}{r_2}(a+b\t{y})(\frac{r_2}{r_1}+b).
\end{gather}
Since $T^{(a,b)}$ is an involution on $\Os$, $\t{y}\in O_3$.  Recall that, by Definition \ref{def-polymer involution}, $h$ maps $O_3\rightarrow (0,\infty)$.  Therefore $O_3\subset W_{a,b}^{+}$ and the four numbers  $r_1$, $r_2$, $\t{y}$, and $h(\t{y})=a+b\t{y}$ are all positive.   \eqref{7-3} now gives (i).

By Lemma \ref{lemma TFAE}, for all $(r_1,r_2)\in O_1\times O_2$ the mapping $O_3\owns y\mapsto T_2^{(a,b)}(r_1,r_2,y)=y(\frac{r_2}{r_1}+b)+a$ is injective.  Therefore $\frac{r_2}{r_1}+b$ does not vanish for any $(r_1,r_2)\in O_1\times O_2$.  Thus, by (i) 
\begin{equation}
O_1\times O_2\subset \left(W_{a,b}^{+}\times V_a^{+}\right)\cup \left(W_{a,b}^{-}\times V_a^{-}\right).  \label{Os constraint}
\end{equation}
If $O_1\cap W_{a,b}^{+}=\emptyset$, then by \eqref{Os constraint} $O_1\times O_2\subset W_{a,b}^{-}\times V_a^{-}$.  In this case $\Os\subset D_{a,b}^{-}$.  On the other hand, if $O_1\cap W_{a,b}^{+}\neq \emptyset$ then there exists $r_1\in O_1$ such that  $a+br_1>0$.  By (i), $r_2-a>0$ for all $r_2\in O_2$.  Thus $O_2\subset V_a^{+}$.  Now \eqref{Os constraint} implies that $O_1\times O_2\subset W_{a,b}^{+}\times V_a^{+}$ which gives  $\Os\subset D_{a,b}^{+}$, completing the proof.

\end{proof}
\noindent Using \eqref{7-3} one can in fact check that $T^{(a,b)}$ is an involution on both $D^{+}_{a,b}$ and $D^{-}_{a,b}$ assuming they are non-empty.

% \begin{proof}[Proof of Corollary \ref{corollary-Lu}]
% The mapping $F(c,d)=(c+d,\frac{c}{c+d})$ is precisely the mapping in theorem \ref{Lu-thm}.  Now (b) directly follows from (a).  We show this explicitly: "$\Rightarrow$" Assume $(C,D)$ are an independent pair.  Then $(A,B)=F(C,D)$.  By assumption $(A,B)$ is an independent pair of non-degenerate random variables.  Thus $(C,D)$ are also non-degenerate.  Thus part (a) applies and yields the desired statement.  "$\Leftarrow$" Assume there exist positive constants $\lambda_A,\lambda_B,\beta$ such that $(A,B)\sim \text{Ga}(\lambda_A+\lambda_B,\beta)\otimes \text{Be}(\lambda_A,\lambda_B)$.  Let $(C',D') \sim \text{Ga}(\lambda_A,\beta)\otimes \text{Ga}(\lambda_B,\beta)$ be new random variables that are independent of the pair $(A,B)$.  Then $(A',B'):=F(C',D')$ are independent of $(A,B)$ and by (a), $(A',B')\sim \text{Ga}(\lambda_A+\lambda_B,\beta)\otimes \text{Be}(\lambda_A,\lambda_B)$.  Thus $(C,D)=H(A,B)\sim H(A',B')=(C',D')$.
% \end{proof}

The following proposition characterizes $T^{h}$-invariant models corresponding to $h(y)=a+by$ when $(a,b)=(0,1),\, (1,0),\, (1,-1),$ and $(-1,1)$.

\begin{proposition}\label{prop fundamental 1}
For $a,b$ real numbers, let $h(y) = a+by$ and assume $T^{(a,b)}$, as defined in \eqref{T-ab}, is a polymer involution adapted to $h$ on $\Os\subset(0,\infty)^3$. Assume that $(R^1,R^2,Y)$ are non-degenerate independent random variables taking values in $\Os$.
\begin{enumerate}
\item  If $(a,b)=(0,1)$, then $(R^1,R^2,Y)$ is $T^{(0,1)}$-invariant if and only if $\ID$ is distributed as in \eqref{model-IG}
\item If $(a,b)=(1,0)$, then $(R^1,R^2,Y)$ is $T^{(1,0)}$-invariant if and only if $\ID$ is distributed as in \eqref{model-G}
\item If $(a,b)=(1,-1)$, then $(R^1,R^2,Y)$ is $T^{(1,-1)}$-invariant if and only if either $\ID$ or $(R^2,R^1,1-Y)$ is distributed as in \eqref{model-B}
\item If $(a,b)=(-1,1)$, then $(R^1,R^2,Y)$ is $T^{(-1,1)}$-invariant if and only if $\ID$ is distributed as in \eqref{model-IB}.
\end{enumerate}
\end{proposition}

\begin{proof}
Observe that $T^{(a,b)}_3$ has no $y$-dependence.   Thus, by Lemma \ref{lemma TFAE}, $T^{(a,b)}$ is the unique polymer involution adapted to $h$ on $\Os$.
By Lemma \ref{lemma distr. TFAE}, $(R^1,R^2,Y)$ is $T^{(a,b)}$-invariant if and only if the following two properties hold:  
\begin{itemize}
\item[(i)] $\frac{R^2}{R^1}$ is independent of $T^{(a,b)}_3(R^1,R^2)$.
\item[(ii)] $Y \stackrel{d}{=} T^{(a,b)}_3(R^1,R^2)$.
\end{itemize}
Recall that
\[
T_3^{(a,b)}(R^1,R^2)=\frac{R^1(R^2-a)}{R^2+bR^1}.
\]

We now prove (a).
Put $(A,B):=\left((R^1)^{-1},(R^2)^{-1}\right)$.  Then $(A,B)$ are non-degenerate independent positive random variables.  Now 
\[\frac{R^2}{R^1}=\frac{A}{B}\quad \text{and}\quad T^{(0,1)}_3(R^1,R^2)=(A+B)^{-1}.\]
So (i) holds if and only if $A/(A+B)=\left(1+B/A\right)^{-1}$ is independent of $A+B$.  By Theorem \ref{thm-lukacs}
this occurs if and only if there exist positive constants $\lambda_A,\lambda_B,\beta$ such that $(A,B)\sim \text{Ga}(\lambda_A,\beta)\otimes \text{Ga}(\lambda_B,\beta)$. In such a case, $A+B=C\sim \text{Ga}(\lambda_A+\lambda_B,\beta)$. Thus $T^{(0,1)}_3(R^1,R^2)=(A+B)^{-1}\sim \text{Ga}^{-1}(\lambda_A+\lambda_B,\beta)$. Now put $(\mu,\lambda)=(\lambda_A+\lambda_B,\lambda_B)$ and use (ii) to get $\ID \sim \text{Ga}^{-1}(\mu-\lambda,\beta)\otimes \text{Ga}^{-1}(\lambda,\beta)\otimes \text{Ga}^{-1}(\mu,\beta)$.  This completes the proof of (a).

We now prove (b). Notice that $D_{1,0}^{-}=\emptyset$. Therefore by Lemma \ref{lemma-involution-properties} we have that $\ID$ takes values in $D_{1,0}^{+}=(0,\infty)\times (1,\infty)\times (0,\infty)$.  Put $(A,B):=\left(R^1,(R^2)^{-1}\right)$.  Then $(A,B)$ are non-degenerate independent random variables taking values in $(0,\infty)\times (0,1)$.  
Now \[\frac{R^2}{R^1}=\frac{1}{AB}\quad \text{and}\quad T^{(1,0)}_3(R^1,R^2)=A(1-B).\] So (i) holds if and only if $AB$ is independent of $A(1-B)$.  By Corollary \ref{corollary-Lu},
this occurs if and only if there exist positive constants $\lambda_A,\lambda_B,\beta$ such that $(A,B)\sim \text{Ga}(\lambda_A+\lambda_B,\beta)\otimes \text{Be}(\lambda_A,\lambda_B)$. In such a case, $T^{(1,0)}_3(R^1,R^2)=A(1-B)=D\sim \text{Ga}(\lambda_B,\beta)$.
Now put $(\mu,\lambda)=(\lambda_B,\lambda_A)$ and use (ii) to get $\ID \sim  \text{Ga}(\mu+\lambda,\beta)\otimes \text{Be}^{-1}(\lambda,\mu)\otimes \text{Ga}(\mu,\beta)$. This completes the proof of (b).

We now prove (c).  By Lemma \ref{lemma-involution-properties}, $\ID$ either takes values in 
\[
D_{1,-1}^{+}=(0,1)\times (1,\infty)\times (0,1) \quad \text{ or } \quad D_{1,-1}^{-}=(1,\infty)\times (0,1)\times (0,1).
\]
First consider the case when $\ID$ takes values in $D^{+}_{1,-1}$.  Put $(A,B):=((R^2)^{-1},R^1)$.  Then $(A,B)$ are non-degenerate independent random variables, both taking values in $(0,1)$.  Now
\[
\frac{R^2}{R^1}=\frac{1}{AB}\quad \text{and}\quad
T^{(1,-1)}_3(R^1,R^2)=1-\frac{1-B}{1-AB}.
\]
So (i) holds if and only if $1-AB$ is independent of $(1-B)/(1-AB)$.  By Theorem \ref{thm-SeWe},
this occurs if and only if there exist positive constants $p,q,r$ such that
$(A,B)\sim \text{Be}(p,q)\otimes \text{Be}(p+q,r).$
In such a case,
$1-T^{(1,-1)}_3(R^1,R^2)=(1-B)/(1-AB)=C\sim \text{Be}(r,q).$
Thus
$T^{(1,-1)}_3(R^1,R^2)\sim 1-\text{Be}(r,q)=\text{Be}(q,r)$.
Now put $(\mu,\lambda,\beta)=(q,p,r)$ and use (ii) to get 
$\ID \sim  \text{Be}(\mu+\lambda,\beta)\otimes \text{Be}^{-1}(\lambda,\mu)\otimes \text{Be}(\mu,\beta)$.

In the case where $\ID$ takes values in $D^-_{1,-1}$, applying the reflection procedure as in \eqref{equation T hat}, one can check that $\widehat{T}=T^{(1,-1)}$ and the resulting random variables $(\widehat{R}^1,\widehat{R}^2,\widehat{Y})=(R^2,R^1,1-Y)$ take values in $D^+_{1,-1}$. By the first case, we are done. This completes the proof of (c).
  
We now prove (d).  Notice that $D^{-}_{-1,1}=\emptyset$. Therefore by Lemma \ref{lemma-involution-properties} $\ID$ must take values in $D^{+}_{-1,1}=(1,\infty)\times (0,\infty)\times (1,\infty)$.  Put $(A,B):=(1-(R^1)^{-1},\,1-(R^2+1)^{-1})$.  Then $(A,B)$ are non-degenerate independent random variables, both taking values in $(0,1)$. Therefore
\[
\left(1+\frac{R^2}{R^1}\right)^{-1}=\frac{1-B}{1-AB}\quad \text{and} \quad
T^{(-1,1)}_3(R^1,R^2)=\frac{1}{1-AB}.
\]
So (i) holds if and only if $(1-B)/(1-AB)$ is independent of $1-AB$.  By Theorem \ref{thm-SeWe}, this occurs if and only if there exist positive constants $p,q,r$ such that $(A,B)\sim \text{Be}(p,q)\otimes \text{Be}(p+q,r)$.
In such a case, $T^{(-1,1)}_3(R^1,R^2)=(1-AB)^{-1}=D^{-1}\sim \text{Be}^{-1}(r+q,p)$.
Now put $(\mu,\lambda,\beta)=(r+q,r,p)$ and use (ii) to get $\ID \sim   \text{Be}^{-1}(\mu-\lambda,\beta)\otimes (\text{Be}^{-1}(\lambda,\beta+\mu-\lambda)-1) \otimes \text{Be}^{-1}(\mu,\beta)$.  This completes the proof of (d).
\end{proof}

We now prove the first main result.
\begin{proof}[Proof of Theorem \ref{thm-classify-linear}]
When $h(y)=a+by$, for all fixed $(r_1,r_2)\in O_1\times O_2$  the mapping $y\mapsto T_2^h(r_1,r_2,y)=y(\frac{r_2}{r_1}+b)+a$ is injective whenever $b\geq 0$.  In the case $b<0$ and $a>0$ this injectivity follows from the assumption $-b\not\in \{\frac{y}{x}:(x,y)\in O_1\times O_2\}$.  Therefore the conditions of Proposition \ref{prop-invar-tinvar-eq}-(a) are satisfied in all cases, which gives the existence of a unique polymer involution $T^h$ adapted to $h(y)=a+by$ such that $(R^1,R^2,Y)$ is $T^h$-invariant.  By \eqref{T=G tensor id...}, $T^h=T^{(a,b)}$  as defined in \eqref{T-ab}.  Now applying Lemma \ref{lem-scaling} then Proposition \ref{prop fundamental 1} completes the proof. 
\end{proof}

\bibliographystyle{plain}
\bibliography{11182016pft.bbl}

\end{document}